\documentclass[12pt]{amsart}
\usepackage[utf8]{inputenc}

\usepackage{amsthm}

\usepackage[english]{babel}
\usepackage{url,graphics}
\usepackage[dvips]{graphicx,epsfig}
\usepackage[T2A]{fontenc}
\usepackage{xcolor}
\usepackage{amssymb, amsmath,amsfonts,array,amscd}
\usepackage{geometry}
\usepackage{graphics}
\usepackage{lineno}
\usepackage{mathrsfs}

\newtheorem{rem}{Remark}
\newtheorem{thm}{Theorem}
\newtheorem*{conjecture}{Conjecture}

\newtheorem{lem}{Lemma}
\newtheorem{prop}{Proposition}

\newcommand{\R}{\mathbb{R}}

\begin{document}
\title{Extremal random beta polytopes}
\author{Ekaterina Simarova}
 \address{ St. Petersburg State University, Faculty of Mathematics and Mechanics, Universitetsky pr. 28, Stary Peterhof 198504, Russian Federation}
\address{Leonhard Euler International Mathematical Institute (SPbU Department),14th Line 29B, Vasilyevsky Island, St. Petersburg, 199178, Russian Federation}
\email{katerina.1.14@mail.ru}
\maketitle
\begin{abstract}
The convex hull of several i.i.d. beta distributed random vectors in $\mathbb R^d$ is called the random beta polytope. Recently, the expected values of their intrinsic volumes, number of faces, normal and tangent angles and other quantities have been calculated, explicitly and asymptotically.  In this paper, we aim to investigate the asymptotic behavior of the beta polytopes with extremal intrinsic volumes. We suggest a  conjecture and solve it in dimension 2. To this end, we obtain some general limit relation for a wide class of $U$-$\max$ statistics  whose kernels include the perimeter and the area of the convex hull of the arguments.

\end{abstract}

\section{Introduction}

\subsection{Beta-polytopes}
Let $U_1, \ldots , U_n$ be random points in $\R^d$ chosen independently with respect to the beta distribution with parameter $\beta>-1$ whose probability density function is defined as
\begin{align*}
    p_{d, \beta}(x)=c_{d,\beta} \cdot (1-\|x\|^2)^{\beta}\cdot{\bf 1}_{\mathbb B^d}(x),\quad \text{where}\quad c_{d,\beta}=\frac{\Gamma(\frac{d}{2}+1+\beta)}{\pi^{\frac{d}{2}}\Gamma(\beta+1)},
\end{align*}
$\mathbb B^d:=\{x\in\R^d:\|x\|\leq 1\}$ is the unit ball, and $\|\cdot\|$ denotes the Euclidean norm in $\R^d$. Their convex hull 
$
    [U_1, \ldots , U_n]
$
is called the \emph{random beta polytope}.

In recent years, there has been an increased interest in the study of the average geometric characteristics of the beta polytopes such as intrinsic volumes, number of faces, normal and tangent angles etc., see~\cite{n29, n30, n32}.

In our paper, instead of the average characteristics we aim to investigate  the extremal ones. To this end, consider some large integer $N>n$ and let $U_1, \ldots , U_N\in\R^d$ be independent beta distributed random vectors defined as above. Given this, we can construct $N\choose n$ random beta polytopes of the form $[U_{i_1}, \ldots , U_{i_n}]$, where $1\leq i_1<i_2<\dots<i_n\leq N$. Now consider some geometric characteristic, say the $m$-th intrinsic volume $v_m(\cdot)$, and choose the random beta polytope which maximizes it:
\begin{align*}
    v_m([U_{i_1}, \ldots , U_{i_n}])\mapsto\max.
\end{align*}
It is not hard to show that  as $N\to\infty$ this maximum converges in probability to the $m$-th intrinsic volume of the polytope which maximizes it among all polytopes lying inside the unit ball $\mathbb B^d$ (the support of the beta distribution):
\begin{align}\label{1911}
    \max\limits_{1\leq i_1<\ldots<i_n\leq N}v_m([U_{i_1}, \ldots , U_{i_n}])\overset{P}{\underset{N\to\infty}\longrightarrow}\max_{x_1,\dots,x_n\in\mathbb B^{d}} v_m([x_1,\dots,x_n]).
\end{align}
Although it is clear that such polytope exists and its vertices lie on the unit sphere $\mathbb S^{d-1}$, its exact shape  is known only for few values of $d$ and $n$ even when $m=d$, see~\cite{BH70, HL16}.

Our goal is to get a refinement of~\eqref{1911}.
Specifically, we believe that the following statement is true:
\begin{conjecture}
For any fixed  $d,n\in\mathbb N$, $\beta>-1$ and $m\in\{0,1,\dots,d\}$ there exist positive numbers 
\begin{align*}
    A = A(d,m,n,\beta),\quad  B = B(d,m,n,\beta), \quad  C = C(d,m,n,\beta)
\end{align*}
such that
\begin{align*}
    \lim_{N \rightarrow \infty} \mathbb{P}\Big[ N^{A}\cdot\Big(\max_{x_1,\dots,x_n\in\mathbb B^{d}} v_m([x_1,\dots,x_n])-\max\limits_{1\leq i_1<\ldots<i_n\leq N}v_m([U_{i_1}, \ldots , U_{i_n}])\Big) \le t\Big]
    \\
    =1 -e^{-B\cdot t^C}.
\end{align*}
\end{conjecture}
Our first theorem solves the conjecture in dimension $d=2$ and gives the exact values for $A,B,C$. In this case, the first intrinsic volume coincides with the semi-perimeter and the second one -- with the area. Note that among all $n$-gones with vertices in the unit disk, the regular $n$-gon inscribed in the unit circle maximizes the area and the perimeter which in this case are equal to $2n\sin{\frac{\pi}{n}}$ and   $\frac{n}{2}\sin{\frac{2\pi}{n}}.$ 

We always assume that $d=2$, from now on. In this case, the density of the beta distribution reduces to
\begin{align}\label{635}
    p_{2, \beta}(x)=\frac{\beta+1}{\pi} \cdot (1-\|x\|^2)^{\beta}\cdot{\bf 1}_{\mathbb B^2}(x).
\end{align}
\begin{thm}\label{1142}
Let $U_1, \ldots , U_N\in\R^2$ be independent beta distributed random vectors with the parameter $\beta>-1$. Then for any $t>0$ we have
\begin{align}\label{1132}
\lim_{N \rightarrow \infty} \mathbb{P}\bigg[N^{\frac{n}{n(\beta+3/2)-1/2}} \left(2n\sin{\frac{\pi}{n}}-\max\limits_{1\leq i_1<\ldots<i_n\leq N}\mathrm{per}([U_{i_1}, \ldots , U_{i_n}])\right)\le t\bigg]
\\\notag
=1 -\exp\bigg[-K_n\frac{(n-1)!}{\sqrt{n}\left(\sin{\frac{\pi}{n}}\right)^{(\beta+3/2)n-1/2}2^{(\beta+1/2)n+1/2}}\cdot t^{(\beta+3/2)n-1/2}\bigg],
\end{align}
and
\begin{align}\label{1133}
\lim_{N \rightarrow \infty} \mathbb{P}\bigg[N^{\frac{n}{n(\beta+3/2)-1/2}} \left(\frac n2\sin{\frac{2\pi}{n}}-\max\limits_{1\leq i_1<\ldots<i_n\leq N}\mathrm{area}([U_{i_1}, \ldots , U_{i_n}])\right)\le t\bigg]
\\\notag
=1 -\exp\bigg[-K_n\frac{(n-1)!\,2^{\frac{n-1}{2}}}{\sqrt{n}\left(\sin{\frac{2\pi}{n}}\right)^{(\beta+3/2)n-1/2}}\cdot t^{(\beta+3/2)n-1/2}\bigg],
\end{align}
where
\begin{align}\label{749}
    K_n= \frac{  2^{(\beta+1/2)n+1/2} \left(\Gamma(\beta+2)\right)^n }{\pi^{\frac{n-1}{2}}n!\,
    \Gamma\left((\beta+\frac32)n+\frac12\right) },
\end{align}
$\mathrm{per}(\cdot), \mathrm{area(\cdot)}$ denote the perimeter and the area, and the rate of convergence is $O(N^{-\frac{1}{(2\beta+3)n-1}})$.

\end{thm}

Theorem~\ref{1142} is a corollary of more general and technical Theorem~\ref{t4} which appears in Section \ref{ss2}. To formulate it, we need to introduce some quantities which generalize the left-hand side of~\eqref{1911} and are called {$U$-$\max$ statistics.}

\subsection{$U$-$\max$ statistics}
\label{ss1}

Let $ \xi_1, \xi_2, \dots $  be a sequence of independent identically distributed random elements taking values in a measurable space $(\mathfrak{X}, \mathfrak{A})$.  Consider some   Borel function
\begin{align*}
    f:{\mathfrak{X}}^n\mapsto \R^1
\end{align*}
which is invariant with respect to the permutations of its arguments. Such a function we call a kernel of degree $n$. Now for $N\geq n$ define the $ U$-$\max$ statistic with the kernel $f$ as
\begin{align*}
\max\limits_{1\leq i_1<\ldots<i_n\leq N} f(\xi_{i_1}, \ldots, \xi_{i_n}).
\end{align*}
The $U$-$\min$ statistics are defined in a similar way.

Initially $U$-$\max$ statistics have been introduced by Lao and Meyer~\cite{n5},~\cite{n6},~\cite{n9} as the extreme counterparts of $U$-statistics which have been studied in details in many publications (see, e.g., \cite{n1}, \cite{n2}, \cite{n4}). 

There are very few results on asymptotic behaviour of $U$-$\max$ statistics defined on $\mathfrak{X}=\mathbb{R}^d$ of arbitrary dimension $d$. The only exceptions known to us is the paper \cite{n5} where  some kernels of degree 2 defined on the set  of points of the unit ball $\mathbb B^d$ are considered. It is also worth noting that such a popular object from Stochastic geometry as the diameter of a  set of random points (see, e.g., \cite{n25}, \cite{n26}, \cite{n27}, \cite{n28}) formally can be regarded as a $U$-$\max$ statistic with the kernel $f(x_1,x_2)=\|x_1-x_2\|$.

All these results correspond to the kernels of degree 2. More complicated kernels appear in the case when $\mathfrak{X}=\mathbb S^1$, the unit circle. 
Lao and Mayer~\cite{n5} considered  the area and the perimeter of random triangles.  Koroleva and Nikitin~\cite{n14} considered $U$-$\max$ statistics of more complicated nature. In particular, they considered the maximal perimeter among all  perimeters of convex $ n$-gons where random vertices  are chosen from $ N $  points independently and uniformly distributed on the unit circle. This was generalized in another direction in~\cite{n21} and~\cite{n22} where a generalized perimeter of a random convex polygon was considered.
The area and the perimeter of inscribed polygons with weaker conditions on  the distribution of  vertices were studied in~\cite{n20}. The paper~\cite{n24} generalizes previous results and contains  the general formulas for the limit behaviour of $U$-$\max$ statistics for a wide class of  distributions of points on the unit circle and for a wide and general class of smooth kernels.

In the next section, we formulate our main result which deals with essentialy the same wide class of the kernels as~\cite{n24} and which
implies Theorem~\ref{1142}. Sections~\ref{ss3} and~\ref{1240} contain the proofs.

\section{Main result}
\label{ss2}

In this section we would like to generalize Theorem~\ref{1142} to a much wider class of kernels. To this end, let us first introduce some notation and conditions.

Since the support of the beta distribution is the unit ball, we consider the kernels mapping as follows:
\begin{align*}
    f:\left(\mathbb B^2\right)^n \rightarrow \mathbb{R} \cup \{-\infty\}.
\end{align*}

To avoid trivialities, in what follows, we always assume that $n \ge 2.$

Similarly to \cite{n24}, we denote by $ \varphi_i $ the angle between the vectors $ OU_1 $ and  $ OU_{i } $ (taken counterclockwise), where $O$ denotes the origin:
\begin{align}
\label{bb1}
\varphi_i=\angle U_{1}OU_{i},\quad i=2, \ldots, n.
\end{align}
We  call such angles central.
All angles that appear in this paper and algebraic operations involving them are  considered modulo $ 2 \pi, $  unless otherwise stated.

Also denote by $r_i$ the distance between $O$ and $U_i:$
\begin{align}
\label{rr1}
r_i=\|OU_i\|, \quad  i=1, \ldots, n.
\end{align}
 In other words, we consider the polar coordinate system where the point $U_1$ has coordinates $(0, r_1),$ and for $i=2, \ldots, n,$ the point $U_i$ has coordinates $(\varphi_i, r_i)$.

Now let us impose some conditions on the kernel $f$. They are similar to the ones from~\cite{n24} with minor changes.

{\bf Conditions on the kernel $ f $:}
\label{tt2}
\begin{itemize}
\item[\bf A1]  $ f $ is invariant with respect to rotations, that is, there exists a function
\begin{align*}
    h(x_1,\dots,x_{2n-1}):[0,2\pi)^{n-1}\times (0,1]^{n} \rightarrow \mathbb{R} \cup \{-\infty\}
\end{align*}
such that
\begin{align*}
 f(U_1, \ldots, U_n)= h(\varphi_2, \ldots, \varphi_n, r_1, \ldots, r_n),
 \end{align*}
where $\varphi_i$ and $r_i$ are defined in~\eqref{bb1} and~\eqref{rr1};
\item[\bf A2] $ f $ is invariant with respect to the  permutations of its arguments;
\item[\bf A3] $ h $ is continuous and can be continuously extended to a function $h:[0,2\pi]^{n-1} \times [0,1]^n \rightarrow \mathbb{R} \cup \{-\infty\}$;
\item[\bf A4] $ h $ attains its maximal value $ M $  only at a finite number of points $V_1, \ldots, V_k$ and also we assume that these points satisfy $ V_1, \ldots, V_k \in (0, 2 \pi)^{n-1}\times \{1\}^n$ which means that the arguments maximizing $f$ are different and lie on the unit circle $\mathbb S^2$;

\item[\bf A5] there exists $ \delta> 0 $ such that  function $ h $ is three times continuously differentiable in the $ \delta$-neighborhood of any maximum point $ V_1,\dots, V_k $;
\item[\bf A6] for any   $ i \in \{1, \ldots, k \} $, the sub-hessian of $h$ at $V_i$ corresponding to the first $n-1$ arguments,
$$G_i:=
\begin{pmatrix}
\frac{\partial^2 h(V_i)}{\partial^2 x_1} & \frac{\partial^2 h(V_i)}{\partial x_1 \partial x_2 } & \ldots & \frac{\partial^2 h(V_i)}{\partial x_1 \partial x_{n-1}} \\
\frac{\partial^2 h(V_i)}{\partial x_1 \partial x_2} & \frac{\partial^2 h(V_i)}{\partial^2 x_2 } & \ldots & \frac{\partial^2 h(V_i)}{\partial x_2 \partial x_{n-1}} \\
\vdots & \vdots & \ddots & \vdots \\
\frac{\partial^2 h(V_i)}{\partial x_{n-1} \partial x_1} & \frac{\partial^2 h(V_i)}{ \partial  x_{n-1}\partial x_2 } & \ldots & \frac{\partial^2 h(V_i)}{\partial^2 x_{n-1}} \\
\end{pmatrix},
 $$
is non-degenerate: $\det G_i\ne 0$;
\item[\bf A7] for any   $ i \in \{1, \ldots, k \} $, all partial derivatives of $h$ at $V_i$ with respect to the last $n$ arguments are non-zero: 
\begin{align*}
    \frac{\partial h(V_i)}{\partial x_j} \ne 0\quad \textup{for}\quad j=n,\dots,2n-1.
\end{align*}

\end{itemize}
\medskip

Now we are ready to formulate our main result.
\begin{thm}
\label{t4}
Suppose that the  kernel $f:\left(\mathbb B^2\right)^n \rightarrow \mathbb{R} \cup \{-\infty\}$  satisfies Conditions~{\bf A1--A7} and let $U_1, \ldots , U_N\in\R^2$ be independent beta distributed random vectors with the probability density function defined in~\eqref{635}. Then for every $t>0$ as $N\to\infty$,
\begin{align*}
 \mathbb{P}\bigg[N^{\frac{n}{n(\beta+3/2)-1/2}} \left(\max_{x_1,\dots,x_n\in\mathbb B^2}f(x_1,\dots,x_n)-\max\limits_{1\leq i_1<\ldots<i_n\leq N}f(U_{i_1}, \ldots , U_{i_n})\right)\le t\bigg]
\\
=\bigg(1 -\exp\big[{- K_n\cdot I[V_1,\dots,V_k]\cdot t^{n(\beta+3/2)-1/2}}\big]\bigg)\bigg(1+O(N^{-\frac{1}{(2\beta+3)n-1}})\bigg),
\end{align*}
where 
\begin{align}\label{1407}
    I[V_1,\dots,V_k]:=\sum_{i=1}^k\frac{1}{\sqrt{\det(-G_i)}\prod\limits_{j=1}^n\left( \frac{\partial h(V_i)}{\partial x_{n-1+j}}\right)^{\beta+1}},
\end{align}
$V_1,\dots, V_k$ are the points from Condition~{\bf A5}, $h$ is from~{\bf A1}, $G_i$'s are from~{\bf A6}, and $K_n$ is defined in~\eqref{749}.
\end{thm}

It can be shown that it is possible to generalize Theorem~\ref{t4} by supposing that $U_1,\dots, U_n$ are independently and identically distributed with respect to the common probability density 
\begin{align}\label{1042}
    p(\varphi,r)\cdot\|1-r^2\|^\beta,
\end{align}
where $p$ is supported inside $\mathbb B^2$ and continuous inside $S^{1} \times (\delta,1]$ for some $\delta<1.$ 
In this case, we divide $K_n$ by the normalizing constant $\left(\frac{\beta+1}{\pi}\right)^{n}$ for the joint density of $n$ independent beta distributed  points and multiply by
\begin{align*}
    \frac{1}{2\pi}\int_{0}^{2\pi}p(\varphi_1,1)\prod_{j=1}^{n-1} p(\varphi_1+V_i^j,1)\, d\varphi_1.
\end{align*}

It is straightforward that the distribution on the unit sphere with density $p(\varphi,1)$ considered in~\cite{n24} can be regarded as a weak limit of the  distribution defined by~\eqref{1042} with  $\beta$ tending to $-1$. Therefore by justifying the limit transition we can deduce the corresponding result from ~\cite{n24}. We prefer to skip the details here.

\section{Proofs of Theorem~\ref{1142},~\ref{t4}}\label{ss3}

\subsection{Proof of Theorem~\ref{1142}}
First let us deduce Theorem~\ref{1142} from Theorem~\ref{t4}. 
Consider the case when $f(U_1, \ldots, U_n)$ is the perimeter of the convex hull of the points $U_1, \ldots, U_n.$ Then the maximal points of the function $f$ correspond to the regular $n$-gon with vertices on the unit circle $\mathbb S^{1}.$ There are $(n-1)!$ maximal points of the function $h.$ Note that 
\begin{align*}
h(\varphi_2, \ldots, \varphi_n, r_1, \ldots, r_n)=\sum\limits_{i=1}^n\sqrt{r_{j_{i+1}}^2+r_{j_i}^2-2r_{j_i}r_{j_{i+1}}\cos{(\varphi_{j_{i+1}}-\varphi_{j_i}})}, \end{align*}
where $(j_2, \ldots, j_n) $ is a permutation of  $(2, \ldots, n)$ such that
\begin{align*}
0=\varphi_{j_1} \le \varphi_{j_2} \le \ldots \le \varphi_{j_n}\le  \varphi_{j_{n+1}}=2\pi, r_{j_{n+1}}=r_{j_1}=r_1.
\end{align*}
Note that the function $h$ is 3 times differentiable in some neighbourhood of the point $$(\frac{2\pi}{n}, \ldots, \frac{2\pi(n-1)}{n},\underbrace{1, \ldots, 1}_n)$$ and all other points which can be obtained from this one by the permutation of the angles.  The determinants of all matrices  $G_i$   are equal to
$2^{1-n}n\left(\sin\frac{\pi}{n}\right)^{n-1}$  (see \cite{n24}).
Also  $\frac{\partial h(V_i)}{\partial x_{n-1+j}}=2\sin{\frac{\pi}{n}}$ for every $i \in \{1, \ldots, (n-1)!\}, j \in \{1, \ldots, n \}.$ 
By Theorem~\ref{t4} we obtain~\eqref{1132}.

Now consider the case when $f(U_1, \ldots, U_n)$ is the area of the convex hull of the points $U_1, \ldots, U_n.$ Then the maximal points of the function $f$ correspond to the regular $n$-gon with vertices on the unit circle $\mathbb S^{1}.$ There are $(n-1)!$ maximal points of the function $h.$  Note that 
\begin{align*}
h(\varphi_2, \ldots, \varphi_n, r_1, \ldots, r_n)=\sum\limits_{i=1}^n \frac{r_{j_i}r_{j_{i+1}}\sin{(\varphi_{j_{i+1}}-\varphi_{j_i}})}{2},
\end{align*}
where $(j_2, \ldots, j_n) $ is a permutation of  $(2, \ldots, n)$ such that
\begin{align*}
0=\varphi_{j_1} \le \varphi_{j_2} \le \ldots \le \varphi_{j_n}\le  \varphi_{j_{n+1}}=2\pi,  r_{j_{n+1}}=r_{j_1}=r_1.
\end{align*}
As in the case for the perimeter, this function is 3  times differentiable in some neighbourhoods of all maximal points.  The determinants of all matrices  $G_i$   are equal to  $2^{1-n}n\left(\sin\frac{2\pi}{n}\right)^{n-1}$  (see \cite{n24}).
Also $\frac{\partial h(V_i)}{\partial x_{n-1+j}}=\sin{\frac{2\pi}{n}}$ for every $i \in \{1, \ldots, (~n~-~1~)~!\}$, $j \in \{1, \ldots, n \}.$ 
By Theorem \ref{t4} we obtain~\eqref{1133}.

\subsection{Proof of Theorem~\ref{t4}}

First let us mention 2 theorems which play a key role in the proof of Theorem \ref{t4}. The first theorem  was proved by Lao and Mayer.
 They used some modification of the statement on the Poisson convergence from the monograph  \cite{n7}.

\begin{thm}\cite{n5} \label{LM}
Let $ \xi_1, \xi_2, \dots, \xi_N $  be a sequence of independent identically distributed random elements taking values in a measurable space $(\mathfrak{X}, \mathfrak{A})$ and function $f(x_1, \dots, x_n)$ be a real-valued symmetric Borel function, $f :{\mathfrak{X}}^n\rightarrow\mathbb{R}.$
Let $$H_N=\max_{1\le i_1<i_2< \ldots <i_n \le N} f(\xi_{i_1}, \ldots, \xi_{i_n})$$ be the $U$-$\max$ statistics  and define for any $z \in \mathbb{R}$  the following quantities:
\begin{align*}
&p_{N,z}=\mathbb{P}\left[f(\xi_1,\ldots,\xi_n)>z\right], \, \, \,
\lambda_{N,z}={ N \choose n} p_{N,z},\\
&\tau_{N,z}(r)=\frac{\mathbb{P}\left[f(\xi_1, \ldots,\xi_n)>z, f(\xi_{1+n-r},\xi_{2+n-r}, \ldots, \xi_{2n-r})>z\right]}{p_{N,z}},
\end{align*}
where  $r \in \{1, \ldots, n-1 \}$.
Then for all $N \ge n$ and for each $z \in \mathbb{R}$  we have
\begin{align}
\nonumber
&|\mathbb{P}\left[H_N \le z\right]-e^{-\lambda_{N,z}}| \\
&\le \left( 1- e^{-\lambda_{N,z}} \right) \cdot \left[ p_{N,z}\left({ N \choose n} - { N-n \choose n} \right)+\sum_{r=1}^{n-1}{ n \choose r} { N-n \choose n-r} \tau_{N,z}(r) \right].
\label{lm1}
\end{align}
\end{thm}

\begin{rem}\cite{n5}
\label{zz2}
 If the sample size $N$ tends to infinity, then the right-hand side in \eqref{lm1} is of  order
$$O\left(p_{N,z}N^{n-1}+
\sum_{r=1}^{n-1}\tau_{N,z}(r)N^{n-r}\right),$$ where for $n>1$  the first term is negligibly small with respect  to the sum.
\end{rem}

The next theorem due to Silverman and Brown~\cite{n8} under additional conditions gives a non-trivial Weibull law in the limit.

\begin{thm} \cite{n8}
\label{SB} Suppose that the conditions of Theorem \textup{\ref{LM}} hold. If for some $T \subset \mathbb{R}$ and some sequence of transformations $z_N: T \rightarrow \mathbb{R}$  the following equalities  hold for each $t \in T$,
\begin{align}
\label{fff1}
&\lim_{N \rightarrow \infty}\lambda_{N, z_N(t)} = \lambda_t >0,\\
\label{fff2}
 &\lim_{N \rightarrow \infty} N^{2n-1} p_{N,z_N(t)}\tau_{N,z_N(t)}(n-1)=0,
\end{align}
 then
\begin{align}
\label{fff3}\lim_{N \rightarrow \infty} \mathbb{P}\left[H_N \le z_N(t)\right]=e^{-\lambda_t}
\end{align}
for all $t \in T.$
\end{thm}
\begin{rem} \cite{n5}
\label{rem2}
Condition \eqref{fff1} implies $p_{N,z}=O(n^{-n}).$ Therefore according to Remark \ref{zz2}   the rate of convergence in \eqref{fff3} is $$O\left(N^{-1}+\sum_{r=1}^{n-1}N^{2n-r}p_{N,z}\tau_{N,z}(r)+\left|e^{-\lambda_{N,z}}-e^{-\lambda_t}\right|\right).$$
Hence, for $n \ge 2$  condition \eqref{fff2} can be replaced by
\begin{align}
\label{zz1}
\lim_{N \rightarrow \infty}N^{2n-r}p_{N,z}\tau_{N,z}(r)=0 \text { for any } r \in \{1, \ldots, n-1\}.
\end{align}
\end{rem}

To apply these two results to proving the theorem consider the following transformation:
$$ z_N(t)=M-tN^{-\frac{n}{\left(\beta+3/2\right)n-1/2}},\quad t>0.$$
The most technical part of the proof is taken out to the following two propositions. Due to their complexity, the proofs are postponed to Section~\ref{1240}.
\begin{prop}
\label{ll1}
 Under the conditions of Theorem \ref{t4} the following relation holds for $\varepsilon \rightarrow 0+$\textup{:} $$  \mathbb{P}\left[ f(U_1, \ldots, U_n) \ge M-\varepsilon \right] = n!\cdot K_n\cdot I[V_1,\dots,V_k]\cdot \varepsilon^{\left(\beta+3/2\right)n-1/2}(1+O(\varepsilon)),$$ where
$K_n$ is defined in~\eqref{749} and $I[V_1,\dots,V_k]$ is defined  in~\eqref{1407}.
\end{prop}

\begin{prop}
\label{ll5}
 For each  $r \in \{1, \ldots, n-1 \}$ we have the following  relation: $$N^{2n-r}\mathbb{P}\left[f(U_1, \ldots, U_n)>z_N(t), f(U_{1+n-r}, \ldots, U_{2n-r})>z_N(t)\right]=O(N^{\frac{-1}{(2\beta+3)n-1}})$$  when $ N \rightarrow  +\infty.$
\end{prop}
Now let us apply these two propositions along with Theorems~\ref{LM},~\ref{SB} to finish the proof of Theorem~\ref{t4}. Consider $\lambda_{N, z_N(t)}$ defined in Theorem \ref{LM}:
$$\lambda_{N,z_N(t)} =\frac{N!}{n! (N-n)!}\mathbb{P}\left[f(U_1, \ldots, U_n)>z_N(t)\right].$$
For brevity to the end of this subsection we write $a\left(\beta,n\right)=\left(\beta+3/2\right)n-1/2.$ We take $\varepsilon=tN^{-\frac{n}{a\left(\beta,n\right)}}, $ then $N^n \varepsilon^{a\left(\beta,n\right)}= t^{a\left(\beta,n\right)}.$
Let us prove the fulfillment of Condition (\ref{fff1}) of  Theorem \ref{SB} (Silverman-Brown Theorem). We write:
\begin{align*}
&\lim_{N \rightarrow \infty} \lambda_{N,z_N(t)} = \lim_{N \rightarrow \infty} \frac{N!}{m! (N-n)!}\mathbb{P}\left[f(U_1, \ldots, U_n)>z_N(t)\right] \\
&=\frac{1}{n!} \lim_{N \rightarrow \infty} \frac{N!}{N^n (N-n)!} N^n \varepsilon^{a\left(\beta,n\right)} \varepsilon^{-a\left(\beta,n\right)} \mathbb{P}\left[f(U_1, \ldots, U_n)>M-\varepsilon\right] \\
&=\frac{ t^{a\left(\beta,n\right)} }{n!}\lim_{N \rightarrow \infty} \left(tN^{-\frac{n}{a\left(\beta,n\right)}}\right)^{-a\left(\beta,n\right)}\mathbb{P}\Big[f(U_1, \ldots, U_n)>M-tN^{-\frac{n}{a\left(\beta,n\right)}}\Big]\\
&=t^{a\left(\beta,n\right)}K_n I[V_1,\dots,V_k] = :\lambda_t>0,
\end{align*}
where in the last line we used Proposition~\ref{ll1}. Now we will prove  Condition (\ref{zz1}) of Remark \ref{rem2} which has the following form:
 $$\lim_{N \rightarrow \infty} N^{2n-r} p_{z_N(t)}\tau_{z_N(t)}(r)=0 \text{ for any } r \in \{1, \ldots, n-1 \}.$$ According to Remark \ref{rem2} Condition (\ref{fff2}) of Theorem \ref{SB} can be replaced by this one.  Proposition \ref{ll5} proves this limiting relation.

Therefore, we may use Theorem \ref{SB}, since all its conditions are verified. Then according to  \eqref{fff3} we obtain
$$\lim_{N \rightarrow \infty} \mathbb{P}\left[H_N \le z_N(t)\right]=e^{-\lambda_t}$$
 for any $t \in T.$
Hence, $$\lim_{N \rightarrow \infty} \mathbb{P}\left[H_N \le M-t N^{-\frac{n}{a\left(\beta,n\right)}}\right]=\exp\left[-t^{a\left(\beta,n\right)}K_n I[V_1,\dots,V_k]\right].$$
Therefore, for any $ t> 0 $ the following relation is valid:
$$\lim_{N \rightarrow \infty} \mathbb{P}\left[ N^{\frac{n}{a\left(\beta,n\right)}} (M-H_N)\le t) \right]=1 -\exp\left[-t^{a\left(\beta,n\right)}K_n I[V_1,\dots,V_k]\right].$$

According to Remark \ref{rem2}, the convergence rate is   $$O\left(N^{-1}+\sum_{r=1}^{n-1}p_{N,z_N(t)}\tau_{N,z_N(t)}(r)N^{2n-r}\right)+O\left(\left|e^{-\lambda_{N,z_N(t)}}-e^{-\lambda_t}\right|\right).$$ By Proposition \ref{ll5} the first part of expression is equal to  $ O(N^{-\frac{1}{(2\beta+3)n-1}}).$   
Also note that 
\begin{align*}
    &\left|e^{-\lambda_{N,z_N(t)}}-e^{-\lambda_t}\right|= O\left(\left|\lambda_{N,z_N(t)}-\lambda_t\right|\right)=\\&O\left(\frac{N!}{N^n(N-n)!}\varepsilon^{-a\left(\beta,n\right)}\mathbb{P}\Big[f(U_1, \ldots, U_n)>M-\varepsilon \Big] -n! K_n I[V_1,\dots,V_k]\right),
\end{align*}
  where $\varepsilon$ is the same as earlier.
  By Proposition 1 it is equal to
  \begin{align*}
   &O\left(n!K_n I[V_1,\dots,V_k]\left((1+O(N^{-1}))(1+O(\varepsilon))-1\right)\right)= O(N^{-1})+O\left(N^{-\frac{n}{\left(\beta+3/2\right)n-1/2}}\right)\\&=o(N^{-\frac{1}{(2\beta+3)n-1}}),
  \end{align*}
and Theorem \ref{t4} follows.

\section{Proofs of Propositions \ref{ll1},~\ref{ll5}}\label{1240}
\subsection{Proof of Proposition \ref{ll1}}
Sometimes for the sake of brevity we will use  notations \begin{align}
\label{bet1}
&\Phi=(\varphi_2, \ldots, \varphi_n, r_1, \ldots, r_n) \in [0, 2\pi)^{n-1}\times (0,1]^n;\\\notag
&\varphi=(\varphi_2, \ldots, \varphi_n) \in [0, 2\pi)^{n-1};\\\notag
& r=(r_1, \ldots, r_n) \in (0,1]^n;\\\notag
&(\varphi_2, \ldots, \varphi_n, r_1, \ldots, r_n)=(\varphi,r)=\Phi.
 \end{align}
For the points $ V_1, \ldots, V_k \in [0, 2 \pi]^{n-1}\times [0,1]^n$ where the maximum $M$ of $h$ is realized we define by $ V ^ j_i $  the $ j$-th component of the point $ V_i.$   From the second part of Condition~A4 we have that for each $i$

\begin{align*}
    V^j_i \in (0, 2 \pi)\quad \text{ for }\quad  j \in \{1, \ldots, n-1\}
\end{align*}
and
\begin{align}
\label{b16}
    V^j_i=1  \quad\text{ for } \quad  j \in \{n, \ldots, 2n-1\}.
\end{align}
Also we  will use  the notation 
\begin{align}
\label{b17}
    &V_i^{\varphi}=(V_i^1, \ldots, V_i^{n-1}) \in (0,2\pi)^{n-1},
    &V_i^{r}=\underbrace{(1,\ldots,1)}_n.
\end{align}

It is clear that
$$\mathbb{P}\left[f(U_1, \ldots, U_n)>z\right] = \mathbb{P}\left[ h(\varphi_2, \ldots, \varphi_{n},r_1, \ldots, r_n\})>z\right],$$
where $ \varphi_i $ are random angles defined in $ \eqref{bb1} $ and $r_i$ are random distances defined in $\eqref{rr1}.$ Further we  deal with   function $ h $ only.

We define for every $\varepsilon>0$ the number
 \begin{align*}
 & S\left(\varepsilon\right) =\min\lbrace\, s \ge 0 \mid \forall x \in [0,2\pi]^{n-1} \times [0,1]^n: M-h(x) \le \varepsilon \\&  \Rightarrow \exists i: \left\|x-V_i\right\| \le s \, \rbrace   
+  \varepsilon^{\frac{1}{3}}.
 \end{align*}
Similarly to \cite{n24} it is easy to show that 
\begin{align}
 \label{lemm1}
 \lim\limits_{\varepsilon \rightarrow +0}S\left(\varepsilon\right) =0.
 \end{align} 
 This implies that for sufficiently small $ \varepsilon $ the following equation holds:
 \begin{equation}
\label{sum}
\mathbb{P}\left[h(\Phi)\ge M-\varepsilon\right] =\sum_{i=1}^k \mathbb{P}\left[h(\Phi) \ge M-\varepsilon, \|V_i-\Phi\|\le S(\varepsilon) \right],
\end{equation}
where $\Phi \in [0, 2\pi]^{n-1} \times [0,1]^n.$

Let us fix some $i \in \{1, \ldots, k\}.$  Assume that the following event happens for some  $\varepsilon>0:$
\begin{align}
\label{f4}
h(\Phi)=h(\varphi_2, \ldots, \varphi_n, r_1, \ldots, r_n) \ge M-\varepsilon,\, \|V_i-\Phi\|\le S(\varepsilon).
\end{align}

By  \eqref{lemm1} there exists $ \varepsilon_0> 0 $ such that $ S(\varepsilon_0) <\min\left(\frac{\delta}{2},1\right), $ where $ \delta $ is the  number from Condition A5. Function $ S(\varepsilon) $ is  non-decreasing, therefore for any positive $ \varepsilon<\varepsilon_0 $ we have
\begin{align}
\label{f7}
S(\varepsilon)<\min\left(\frac{\delta}{2},1\right).
\end{align}
Below we deal with $ \varepsilon<\varepsilon_0 $ only.
Since  function $ h $ is three times continuously differentiable in the $ \delta$-neighborhood of any maximal point, in this neighborhood we consider the Taylor expansion of  function $h$ at the point $ V_i $ with the third order remainder. For this purpose we introduce the  notation:
\begin{align}
\label{f17}
&\alpha_j = \varphi_j-V_i^{j+1} \text{ and } \alpha=(\alpha_2, \ldots, \alpha_{n}),\\
\nonumber
&\rho_j =1 -r_j  \text{ and } \rho=(\rho_1, \ldots, \rho_{n}).
\end{align}
It is clear that $$\left\|(\alpha, \rho)\right\|=\left\|\Phi-V_i\right\|  < \frac{\delta}{2}.$$
Here  $ \alpha $  is an element from $ \mathbb {R}^{n-1} $ which is considered as a difference of two elements from $ \mathbb{R}^{n-1} $ and not as the difference of two sets of angles.
  By  \eqref{lemm1} and Condition  A4 it is the same for small $ \varepsilon. $
  
  We write the Taylor expansion of  function $h$  at the point $ V_i.$
By \eqref{b16} we have
\begin{align}
\nonumber
&h\left(\varphi_2, \ldots, \varphi_n, r_1, \ldots, r_n\right)=h\left(V_i^1+\alpha_2,  \ldots, V_i^{n-1}+\alpha_{n}, 1-\rho_1, \ldots, 1-\rho_n \right)\\
\label{f6}
&=h\left(V_i\right)+\sum_{j=1}^{n-1}  \frac{\partial h\left(V_i\right)}{\partial x_j}\, \alpha_{j+1} -\sum_{j=1}^{n}\frac{\partial h\left(V_i\right)}{\partial x_{n-1+j}}\, \rho_j  + \sum_{1 \le l,s \le n-1}\frac{1}{2} \frac{\partial^2 h\left(V_i\right)}{\partial x_l\partial x_s} \, \alpha_{l+1}\alpha_{s+1} \\
\nonumber
&-\sum_{1 \le l \le n-1, 1 \le s \le n} \frac{\partial^2 h\left(V_i\right)}{\partial x_l\partial x_{n-1+s}} \, \alpha_{l+1} \rho_s +\sum_{1 \le l,s \le n}\frac{1}{2} \frac{\partial^2 h\left(V_i\right)}{\partial x_{n-1+l}\partial x_{n-1+s}} \, \rho_l\rho_s
\\&+\sum_{1 \le l,s,t \le 2n-1}\frac{1}{6} \frac{\partial^3 h\left(V_i+r_{\left(l,s,t\right)}\right)}{\partial x_l \partial x_s \partial x_t}\, y_l y_s y_t,
\nonumber
\end{align}
where $r_{\left(l,s,t\right)} =c_{\left(l,s,t\right)} \cdot \left(\alpha_2, \ldots,\alpha_{n},-\rho_1, \ldots, -\rho_n\right),$ and $c_{\left(l,s,t\right)} \in \left(0,1\right)$ are constants depending on indices $l,s,t$ and on function $h,$ and  $ y_i=\alpha_{i+1}$ when $i<n$ and $y_i=-\rho_{i-n+1}$ when $i\ge n.$  According to \eqref{b17}, $ V_i^{\varphi} $ does not lie on the boundary of the definition domain  of  the continuous function $ h, $ therefore $\frac{\partial h(V_i)}{\partial x_j} =0$ for all $ j \in \{1, \ldots, n-1 \}. $ 

Consider the matrix
\begin{align}
\label{f5}
 A^i=\frac{1}{2}G_i,
\end{align} where $G_i$ are the same as in Condition  A6.
It is clear that the coefficient before $ \alpha_l \alpha_s $ in  (\ref{f6}) is $ a^i_{l, s} $
(the element of the matrix $ A^i$).Thus,
\begin{align*}
&h\left(\Phi\right)=M-\sum_{j=1}^{n}\frac{\partial h\left(V_i\right)}{\partial x_{n-1+j}}\, \rho_j+\sum\limits_{1 \le l,s \le n}a^i_{l,s} \alpha_{l+1}\alpha_{s+1} \\&-\sum_{1 \le l \le n-1, 1 \le s \le n} \frac{\partial^2 h\left(V_i\right)}{\partial x_l\partial x_{n-1+s}} \, \alpha_{l+1} \rho_s +\sum_{1 \le l,s \le n}\frac{1}{2} \frac{\partial^2 h\left(V_i\right)}{\partial x_{n-1+l}\partial x_{n-1+s}} \, \rho_l\rho_s
\\&+\sum_{1 \le l,s,t \le 2n-1}\frac{1}{6} \frac{\partial^3 h\left(V_i+r_{\left(l,s,t\right)}\right)}{\partial x_l \partial x_s \partial x_t}\, y_l y_s y_t
\nonumber
\end{align*}
Therefore,  condition (\ref{f4}) is equivalent to
\begin{align}
\label{f2}
&\sum_{j=1}^{n}\frac{\partial h\left(V_i\right)}{\partial x_{n-1+j}}\, \rho_j-\sum_{1 \le l,s \le n-1}a^i_{l,s} \alpha_{l+1}\alpha_{s+1}- \sum_{1 \le l,s \le n}\frac{1}{2} \frac{\partial^2 h\left(V_i\right)}{\partial x_{n-1+l}\partial x_{n-1+s}} \, \rho_l\rho_s\\
&+
\sum_{1 \le l \le n-1, 1 \le s \le n} \frac{\partial^2 h\left(V_i\right)}{\partial x_l\partial x_{n-1+s}} \, \alpha_{l+1} \rho_s -\frac{1}{6}\sum_{1 \le l,s,t \le 2n-1} \frac{\partial^3 h(V_i+r_{(l,s,t)})}{\partial x_l \partial x_s \partial x_t} y_l y_s y_t \le \varepsilon.
\nonumber
\end{align}

Under  conditions $ (\ref{f4}) $ and $ (\ref{f7}) $, we estimate some second order terms and all  third order terms  in this formula.
Note that $\|\alpha\|, \|\rho\| \le S(\varepsilon),$ therefore there exists constant $M_1>0$ such that 
\begin{align*}
   & \left|\frac{\partial^2 h\left(V_i\right)}{\partial x_l\partial x_{n-1+s}} \, \alpha_{l+1} \rho_s\right| \le M_1S(\varepsilon)\rho_s,
   &\left| \frac{\partial^2 h\left(V_i\right)}{\partial x_{n-1+l}\partial x_{n-1+s}} \, \rho_l\rho_s\right| \le M_1S(\varepsilon)\rho_l.
\end{align*}

Since  functions   $\frac{\partial h(V_i+r)}{\partial x_l \partial x_s \partial x_t} $ are continuous for $|r| \le \frac{\delta}{2},$  there exists $ M_2 $ such that  $\left|\frac{\partial h(V_i+r)}{\partial x_l \partial x_s \partial x_t}\right|$ does not exceed $ M_2 $ for all $ |r| \le \frac{\delta}{2}. $ 
Now we estimate the third order terms of Taylor expansion. For the terms with $l\ge n$ the following estimation holds:
\begin{align}
\label{r2}
\left|\frac{\partial^3 h(V_i+r_{(l,s,t)})}{\partial x_l \partial x_s \partial x_t} y_l y_s y_t \right| \le M_2 y_l y_s y_t \le  M_2 S(\varepsilon)^2 \rho_{l-n+1} \le  M_2  S(\varepsilon)\rho_{l-n+1}.   \end{align}
Similar estimates are fulfilled for the terms with $ s \ge n, t \ge n.$

For the terms with $l,s,t<n$ the following inequality holds true:
\begin{align}
&\left| \frac{\partial^3 h\left(V_i+r_{\left(l,s,t\right)}\right)}{\partial x_l \partial x_s \partial x_t} y_l y_s y_t\right|=\left| \frac{\partial^3 h\left(V_i+r_{\left(l,s,t\right)}\right)}{\partial x_l \partial x_s \partial x_t} \alpha_{l+1} \alpha_{s+1} \alpha_{t+1}\right|  \le M_2 \left|\alpha_{l+1} \alpha_{s+1} \alpha_{t+1}\right|
\label{f9}
\\
& \le M_2 \frac{\left|\alpha_{l+1}\right|^3+\left|\alpha_{s+1}\right|^3+\left|\alpha_{t+1}\right|^3}{3} \le M_2 \frac{\alpha_{l+1}^2+\alpha_{s+1}^2+\alpha_{t+1}^2}{3} S(\varepsilon).
\nonumber
\end{align}
By \eqref{f2} we get the following inequality for all $\|\rho\|,\|\alpha\|<S(\varepsilon), \rho_j \ge 0, \, j \in \{1, \ldots, n\}$ :
\begin{align*}
&\sum\limits_{j=1}^{n}\frac{\partial h\left(V_i\right)}{\partial x_{n-1+j}}\, \rho_j +\sum_{j=1}^{n}\hat{C_j}S(\varepsilon)\rho_j +M_3S\left(\varepsilon\right) \sum\limits_{s=1}^{n-1}\alpha_{s+1}^2 -\sum\limits_{1 \le l,s \le n}a^i_{l,s}\,\alpha_{l+1}\alpha_{s+1} \\
&\ge
M-h\left(V_i+\alpha\right) \\
 &\ge \sum\limits_{j=1}^{n}\frac{\partial h\left(V_i\right)}{\partial x_{n-1+j}}\, \rho_j -\sum_{j=1}^{n}\hat{C_j}S(\varepsilon)\rho_j-M_3S\left(\varepsilon\right) \sum\limits_{s=1}^{n-1}\alpha_{s+1}^2 -\sum\limits_{1 \le l,s \le n}a^i_{l,s}\, \alpha_{l+1}\alpha_{s+1},
\end{align*}
where $ M_3,\hat{C_j} , j\in\{1, \ldots, n\}$ are some constants obtained by summing the estimates  \eqref{r2}, \eqref{f9}   over all triples  $ (l, s, t). $
According to Condition  A7 and \eqref{b16} we get the inequality 
\begin{align}
\label{b1}
\frac{\partial h\left(V_i\right)}{\partial x_{n-1+j}}>0 \text{ for } i \in \{1, \ldots, k\}, \,  j \in \{1, \ldots, n \}.
\end{align}
Therefore, there exist constants $C_j$ such that the following inequality holds:
\begin{align}
\nonumber
&\sum\limits_{j=1}^{n}\frac{\partial h\left(V_i\right)}{\partial x_{n-1+j}} (1+C_jS(\varepsilon))\, \rho_j  +M_3S\left(\varepsilon\right) \sum\limits_{s=1}^{n-1}\alpha_{s+1}^2 -\sum\limits_{1 \le l,s \le n}a^i_{l,s}\,\alpha_{l+1}\alpha_{s+1} \\
\nonumber
&\ge
M-h\left(V_i+\alpha\right) \\
 &\ge \sum\limits_{j=1}^{n}\frac{\partial h\left(V_i\right)}{\partial x_{n-1+j}} (1-C_j S(\varepsilon))\, \rho_j -M_3S\left(\varepsilon\right) \sum\limits_{s=1}^{n-1}\alpha_{s+1}^2 -\sum\limits_{1 \le l,s \le n}a^i_{l,s}\, \alpha_{l+1}\alpha_{s+1},
 \label{f11}
\end{align}
Denote
\begin{align*}
A^i \left(\varepsilon\right)=
\begin{cases}
&A^i+M_3 S\left(\varepsilon\right)I_{n-1}, \text{ for } \varepsilon \ge 0,  \\
&A^i -M_3 S\left(-\varepsilon\right)I_{n-1}, \text{ for } \varepsilon \le 0,
\end{cases}
\end{align*}
where $ A^{i} $ is the same as in   $ (\ref{f5}), $ and $ I_{n-1} $ is the identity matrix of size $ (n-1)\times (n-1). $
Then  inequality $ (\ref{f11}) $ may be rewritten using the scalar product $ \langle \, \cdot\, ,\, \cdot\, \rangle $ as
\begin{align}
\nonumber
&\mathbb{P}\big[ \sum\limits_{j=1}^{n}\frac{\partial h\left(V_i\right)}{\partial x_{n-1+j}} \left(1+C_jS\left(\varepsilon\right)\right)\rho_j -\langle A^i\left(\,-\varepsilon\right)\alpha, \alpha \rangle \le \varepsilon, \|\alpha\|, \|\rho\|\le S(\varepsilon), \\ &\rho_j \ge 0 \text{ for }  j \in \{1, \ldots, n\}\,\big]
\label{pr}
 \ge \mathbb{P}\big[ 
h\left(\Phi\right) \ge M-\varepsilon, \left\|V_i-\Phi\right\|\le S\left(\varepsilon\right)\big]  \\
\nonumber
&\ge \mathbb{P}\big[ \, \sum\limits_{j=1}^{n}\frac{\partial h\left(V_i\right)}{\partial x_{n-1+j}} (1-C_jS(\varepsilon))\rho_j -\langle A^i\left(\varepsilon\right)\alpha, \alpha \rangle \le \varepsilon, \|\alpha\|, \|\rho\|\le S(\varepsilon), \\&\rho_j \ge 0 \text{ for }  j \in \{1, \ldots, n\}\, \big].
\nonumber
\end{align}

Now we formulate new technical lemma.
\begin{lem}
\label{ll4}
 There exist  constants $\bar{C}$ and $\bar{D} $ such that, for any number $\varepsilon$,\\  $0<\varepsilon<\bar{C},$ if $f(U_1, \ldots, U_n) \ge M-\varepsilon,$ then
there exists  $i \in \{1, \ldots, k \}$ such that  $\|V^{\varphi}_i-\varphi\| \le \bar{D}\sqrt{\varepsilon}, \,  \|V^{r}_i-r\|\le \bar{D}\varepsilon $ where $ \varphi, r $ are defined by  \eqref{bb1}, \eqref{rr1} and $V_i$ is defined  by \eqref{b17} and in Condition {\bf \textup{A4}}.
\end{lem}
\begin{proof}

As it was done in  \cite{n24}, matrices $A^{i}(\varepsilon)$ are negatively defined  for sufficiently small $\varepsilon,$  and $-\langle A^i\left(\varepsilon\right)\alpha, \alpha \rangle \ge 0.$ Therefore, the inequality  $\sum\limits_{j=1}^{n}\frac{\partial h\left(V_i\right)}{\partial x_{n-1+j}} (1-C_jS(\varepsilon))\rho_j -\langle A^i\left(\varepsilon\right)\alpha, \alpha \rangle \le \varepsilon$ together  with condition $\rho_j \ge 0 \text{ for every }  j \in \{1, \ldots, n\} $ and \eqref{b1} implies the following two inequalities:
\begin{align}
\label{b2}
    &\sum\limits_{j=1}^{n}\frac{\partial h\left(V_i\right)}{\partial x_{n-1+j}} (1-C_jS(\varepsilon))\rho_j \le \varepsilon,\\
    \label{b3}
    &-\langle A^i\left(\varepsilon\right)\alpha, \alpha \rangle \le \varepsilon.
\end{align}

Similarly to \cite{n24}, there exist such constants $\bar{C_1},\bar{D_1}>0$ that  the inequality \eqref{b3} implies $\|\alpha\|<\bar{D_1} \sqrt{\varepsilon}$ for every $0 \le \varepsilon<\bar{C_1}$ (see, \cite[Corollary 7.1]{n24} ). By  \eqref{lemm1},\eqref{b1} there are  such constants $\bar{C_2},\bar{D_2}>0$ that  the inequality \eqref{b2} implies $\|\rho\|<\bar{D_2}\varepsilon$  for every $0 \le \varepsilon<\bar{C_2}$. Hence, we can choose $\bar{C}=\min{\left(\bar{C_1}, \bar{C_2}\right)}, \bar{D}=\min{\left(\bar{D_1}, \bar{D_2}\right)},$  and Lemma \ref{ll4} is proved.
\end{proof}
By Lemma \ref{ll4}, we may conclude that for  small $\varepsilon$ the conditions $\|\alpha\|, \|\rho\|<S(\varepsilon)$ from \eqref{pr}  can be deleted without changing the value of these probabilities.

Finally, the proof of Proposition \ref{ll1} follows from the following lemma:
\begin{lem}
\label{ll2}
For some real constants $D_1, \ldots, D_n$ and for $\varepsilon \rightarrow +0$ the following equality holds:
\begin{align}
\label{b4}
    &\mathbb{P}\left[ \sum\limits_{j=1}^{n}\frac{\partial h\left(V_i\right)}{\partial x_{n-1+j}} \left(1+D_jS\left(\varepsilon\right)\right)\rho_j -\langle A^i\left(\,\pm \varepsilon\right)\alpha, \alpha \rangle \le \varepsilon, \rho_j \ge 0 \text{ for }  j \in \{1, \ldots, n\}\,\right] \\
    \nonumber
    &=n! \cdot K_n\cdot I[V_1,\dots,V_k]\cdot\varepsilon^{\left(\beta+3/2\right)n-1/2}   (1+O(\varepsilon)),
\end{align}
where $K_n$ is defined in~\eqref{749} and $I[V_1,\dots,V_k]$ is defined  in~\eqref{1407}.
\end{lem}
\begin{proof}
We have probability density $p(U)=\frac{(\beta+1)}{\pi}\left(1-\|U\|^2\right)^{\beta} {\bf 1}\{\|U\| < 1\},$ where  $U \in \mathbb{R}^2.$ In the polar coordinate system  the probability density  of the point $U =(\phi, r)$ is equal to  $p(U)=\frac{(\beta+1)}{\pi}r\left(1-r^2\right)^{\beta} {\bf 1}\{r \in (0,1)\}.$ Then the following equality holds:
\begin{align*}
    &\mathbb{P} \Big[ \sum\limits_{j=1}^{n}\frac{\partial h\left(V_i\right)}{\partial x_{n-1+j}} \left(1+D_jS\left(\varepsilon\right)\right)\rho_j -\langle A^i\left(\,\pm \varepsilon\right)\alpha, \alpha \rangle \le \varepsilon, \rho_j \ge 0 \text{ for }  1 \le j \le n\, \Big] \\
    &=\underbrace{\int\limits_0^1\ldots\int\limits_0^1}_{n}\underbrace{\int\limits_0^{2\pi}\ldots\int\limits_{0}^{2\pi}}_{n} {\bf 1}\Big{\{} \sum\limits_{j=1}^{n}\frac{\partial h\left(V_i\right)}{\partial x_{n-1+j}} \left(1+D_jS\left(\varepsilon\right)\right)\rho_j -\langle A^i\left(\,\pm \varepsilon\right)\alpha, \alpha \rangle \le \varepsilon \Big{\}} \times \\&  \prod\limits_{j=1}^{n}\left(  \frac{(\beta+1)}{\pi} r_j (1-r^2_j)^{\beta} \right)\, dr_1 \ldots dr_n \, d\phi_1\ldots d\phi_n,
\end{align*}
where $\rho_j, \alpha_j$ are denoted by \eqref{f17} and satisfy the equalities: $\rho_j=1-r_j, j \in \{1, \ldots, n\},\alpha_j=\phi_j-V_i^{j+1}-\phi_1, j \in \{2,\ldots, n\}.$ Now we change variables in this integral: $(r_i, \phi_i) \rightarrow (\rho_i, \alpha_i).$ Then the previous integral is equal to
\begin{align*}
&\underbrace{\int\limits_0^1\ldots\int\limits_0^1}_{n}\int\limits_0^{2\pi}\int\limits_{-V_i^1}^{2\pi-V_i^1}\ldots\int\limits_{-V_i^{n-1}}^{2\pi-V_i^{n-1}} {\bf 1}\Big{\{} \sum\limits_{j=1}^{n} \frac{\partial h\left(V_i\right)}{\partial x_{n-1+j}} \left(1+D_jS\left(\varepsilon\right)\right)\rho_j \\ & -\langle A^i\left(\,\pm \varepsilon\right)\alpha, \alpha \rangle \le \varepsilon \Big{\}}\times \prod\limits_{j=1}^{n}\left( \frac{(\beta+1)}{\pi}\left(1-\rho_j\right) \left(\rho_j \left(2-\rho_j\right)\right)^{\beta} \right) \, d\alpha_n\ldots d\alpha_1 \, d\rho_1 \ldots d\rho_n \\
& = \frac{2(\beta+1)^n}{\pi^{n-1}}\underbrace{\int\limits_0^1\ldots\int\limits_0^1}_{n}\int\limits_{-V_i^1}^{2\pi-V_i^1}\ldots\int\limits_{-V_i^{n-1}}^{2\pi-V_i^{n-1}}{\bf 1}\Big{\{} \sum\limits_{j=1}^{n}\frac{\partial h\left(V_i\right)}{\partial x_{n-1+j}} \left(1+D_jS\left(\varepsilon\right)\right)\rho_j \\&-\langle A^i\left(\,\pm \varepsilon\right)\alpha, \alpha \rangle \le \varepsilon \Big{\}}  \times \prod\limits_{j=1}^{n} \left(\left(1-\rho_j\right) \left(\rho_j \left(2-\rho_j\right)\right)^{\beta} \right) \, d\alpha_n\ldots d\alpha_2 \, d\rho_1 \ldots d\rho_n 
\end{align*}
By Lemma \ref{ll4} the equation $\sum_{j=1}^{n}\left(1+D_jS\left(\varepsilon\right)\right)\rho_i -\langle A^i\left(\,\pm \varepsilon\right)\alpha, \alpha \rangle \le \varepsilon$ implies that $\rho_j<\bar{C_3} \varepsilon.$ Therefore,  $\left(2-\rho_j\right)^{\beta}(1-\rho_i)=2^{\beta}(1+O(\varepsilon))$  for all  $\rho$ such that the integrable expression is greater than 0. Therefore, we can continue the sequence of equalities the following way:

\begin{align}
\label{b5}
& \frac{ 2^{n\beta+1} \, (\beta+1)^n}{\pi^{n-1}} \underbrace{\int\limits_0^1\ldots\int\limits_0^1}_{n}\int\limits_{-V_i^1}^{2\pi-V_i^1}\ldots\int\limits_{-V_i^{n-1}}^{2\pi-V_i^{n-1}}{\bf 1}\Big{\{} \sum\limits_{j=1}^{n}\frac{\partial h\left(V_i\right)}{\partial x_{n-1+j}} \left(1+D_jS\left(\varepsilon\right)\right)\rho_j \\ -&\langle A^i\left(\,\pm \varepsilon\right)\alpha, \alpha \rangle \le \varepsilon \Big{\}}  
\nonumber
\times \prod\limits_{j=1}^{n} \rho_j ^{\beta} \, d\alpha_n  \ldots d\alpha_2\, d\rho_1 \ldots d\rho_n  \cdot \left(1+O\left(\varepsilon\right)\right)= \\
 \nonumber &\frac{ 2^{n\beta+1} \, (\beta+1)^n}{\pi^{n-1}}\left(1+O\left(\varepsilon\right)\right)\int\limits_0^{\varepsilon}\underbrace{\int\limits_0^1\ldots\int\limits_0^1}_{n}{\bf 1}\Big{\{} \sum\limits_{j=1}^{n}\frac{\partial h\left(V_i\right)}{\partial x_{n-1+j}} \left(1+D_jS\left(\varepsilon\right)\right)\rho_j=x \Big{\}} \prod\limits_{j=1}^{n} \rho_j ^{\beta} \\
 \nonumber
    &\times \int\limits_{-V_i^1}^{2\pi-V_i^1}\ldots\int\limits_{-V_i^{n-1}}^{2\pi-V_i^{n-1}}{\bf 1}\Big{\{}  -\langle A^i\left(\,\pm \varepsilon\right)\alpha, \alpha \rangle \le \varepsilon-x \Big{\}}\,   d\alpha_n \ldots d\alpha_2\, d\rho_1 \ldots d\rho_n \,  dx.
\end{align}
Let us consider the integral over the variables $ \alpha_2,\ldots \alpha_n.$ By Lemma \ref{ll4} $\|\alpha\|<\bar{D_1} \sqrt{\varepsilon}.$ Therefore, we can integrate this expression over $\mathbb{R}^{n-1}$ when $\varepsilon$ is small enough. It was shown in \cite{n24} this integral equals $$\frac{\left(\varepsilon \pi (1-x/\varepsilon)\right)^{\frac{n-1}{2}}}{\Gamma\left(\frac{n+1}{2}\right)\sqrt{\det\left(-A^i(\pm \varepsilon)\right)}}.$$ By \eqref{lemm1}, \eqref{f5} and \eqref{b3} it is equal to  $$\frac{\left(\varepsilon \pi (1-x/\varepsilon)\right)^{\frac{n-1}{2}}}{\Gamma\left(\frac{n+1}{2}\right)\sqrt{\det\left(-A^i\right)}}\left(1+O\left(\varepsilon\right)\right) = \frac{\left(2\varepsilon \pi\right)^{\frac{n-1}{2}}  (1-x/\varepsilon)^{\frac{n-1}{2}}}{\Gamma\left(\frac{n+1}{2}\right)\sqrt{\det\left(-G^i\right)}}\left(1+O\left(\varepsilon\right)\right).$$
  Therefore, the integral from the expression \eqref{b5} is equal to the following one:
\begin{align}
\label{b6}
&=\frac{\varepsilon^{\frac{n-1}{2}} 2^{(\beta+1/2)n+1/2} (\beta+1)^{n}}{\pi^{\frac{n-1}{2}}\Gamma\left(\frac{n+1}{2}\right)\sqrt{\det\left(-G^i\right)}} \int\limits_0^{\varepsilon}\underbrace{\int\limits_0^1\ldots\int\limits_0^1}_{n}{\bf 1}\Big{\{} \sum\limits_{j=1}^{n}\frac{\partial h\left(V_i\right)}{\partial x_{n-1+j}} \left(1+D_jS\left(\varepsilon\right)\right)\rho_j=x \Big{\}}  \times \\
 \nonumber
  & \left(1-x/\varepsilon\right)^{\frac{n-1}{2}} \prod\limits_{j=1}^{n} \rho_j ^{\beta} \,  d\rho_1 \ldots d\rho_n \,  dx  \cdot \left(1+O\left(\varepsilon\right)\right).
\end{align}

By \eqref{lemm1} and \eqref{b1} we can integrate over  $[0, +\infty)^{n}.$ Let
\begin{align*}
    &y=\frac{x}{\varepsilon}, &z_j=\frac{\rho_j}{\varepsilon}\text{ for } j \in \{1, \ldots, n\}.
\end{align*}
We change the variables $x, \rho_1, \ldots, \rho_n$ to the variables $y, z_1, \ldots, z_n.$ The integral from \eqref{b6} can be written in the following form:
\begin{align}
\nonumber
    &=\frac{\varepsilon^{\left(\beta+3/2\right)n-1/2} \, 2^{(\beta+1/2)n+1/2} \,  (\beta+1)^{n}}{\pi^{\frac{n-1}{2}}\Gamma\left(\frac{n+1}{2}\right)\sqrt{\det\left(-G^i\right)}} \int\limits_0^{1}\underbrace{\int\limits_0^{+\infty}\ldots\int\limits_0^{+\infty}}_{n}{\bf 1}\Big{\{} \sum\limits_{j=1}^{n}\frac{\partial h\left(V_i\right)}{\partial x_{n-1+j}} \left(1+D_jS\left(\varepsilon\right)\right)z_j=y \Big{\}}   \\
 \label{b14}
  & \times \left(1-y\right)^{\frac{n-1}{2}} \prod\limits_{j=1}^{n} z_j ^{\beta} \, dz_1 \ldots dz_n \,  dy  \cdot \left(1+O\left(\varepsilon\right)\right)\\
  \nonumber
  &=\frac{\varepsilon^{\left(\beta+3/2\right)n-1/2} \, 2^{(\beta+1/2)n+1/2} \,  (\beta+1)^{n}}{\pi^{\frac{n-1}{2}}\Gamma\left(\frac{n+1}{2}\right)\sqrt{\det\left(-G^i\right)}} \underbrace{\int\limits_0^{+\infty}\ldots\int\limits_0^{+\infty}}_{n}{\bf 1}\Big{\{} \sum\limits_{j=1}^{n}\frac{\partial h\left(V_i\right)}{\partial x_{n-1+j}} \left(1+D_jS\left(\varepsilon\right)\right)z_j<1 \Big{\}}  \\
 \nonumber
  &  \times \left(1-\sum\limits_{j=1}^{n}\frac{\partial h\left(V_i\right)}{\partial x_{n-1+j}} \left(1+D_jS\left(\varepsilon\right)\right)z_j\right)^{\frac{n-1}{2}} \prod\limits_{j=1}^{n} z_j ^{\beta} \, dz_1 \ldots dz_n \,   \cdot \left(1+O\left(\varepsilon\right)\right).
\end{align}
The proof of  \eqref{b4} follows from the following lemma:
\begin{lem}
\label{ll3}
Suppose that $a_j>0$ for every $j \in \{1, \ldots, n\}.$ Then the following equality holds  
\begin{align}
\label{b7}
    &\underbrace{\int\limits_0^{+\infty}\ldots\int\limits_0^{+\infty}}_{n}{\bf 1}\Big{\{} \sum\limits_{j=1}^{n} a_j z_j<1 \Big{\}}   \left(1-\sum\limits_{j=1}^{n}a_j z_j\right)^{\frac{n-1}{2}} \prod\limits_{j=1}^{n} z_j ^{\beta} \, dz_1 \ldots dz_n \\
    \nonumber&=\frac{\Gamma\left(\frac{n+1}{2}\right)\left(\Gamma\left(\beta+1\right)\right)^{n}}{\Gamma\left(\frac{n-1}{2}+n(\beta+1)+1\right)} \prod\limits_{j=1}^{n}a_j^{-1-\beta}.
\end{align}
\end{lem}
\begin{proof}
Denote by $ t_j=a_jz_j$ the new variables in the integral from \eqref{b7}. We obtain that \eqref{b7} is equal to 
\begin{align}
\label{b13}
     \prod\limits_{j=1}^{n}a_j^{-1-\beta} \underbrace{\int\limits_0^{+\infty}\ldots\int\limits_0^{+\infty}}_{n}{\bf 1}\Big{\{} \sum\limits_{j=1}^{n} t_j<1 \Big{\}}   \left(1-\sum\limits_{j=1}^{n}t_j\right)^{\frac{n-1}{2}} \prod\limits_{j=1}^{n} t_j ^{\beta} \, dt_1 \ldots dt_n
\end{align}
We will prove that the following equality holds for every $l \in \{ 1, \ldots, n\}:$
\begin{align}
\nonumber
 &\underbrace{\int\limits_0^{+\infty}\ldots\int\limits_0^{+\infty}}_{n}{\bf 1}\Big{\{} \sum\limits_{j=1}^{n} t_j<1 \Big{\}}   \left(1-\sum\limits_{j=1}^{n}t_j\right)^{\frac{n-1}{2}} \prod\limits_{j=1}^{n} t_j ^{\beta} \, dt_1 \ldots dt_n \\
&=\frac{\Gamma\left(\frac{n+1}{2}\right)\Gamma\left(\beta+1\right)^{l}}{\Gamma\left(\frac{n-1}{2}+l(\beta+1)+1\right)} \times
  \label{b8}
 \underbrace{\int\limits_0^{+\infty}\ldots\int\limits_0^{+\infty}}_{n-l}{\bf 1}\Big{\{} \sum\limits_{j=1}^{n-l} t_j<1 \Big{\}}\\ \nonumber &   \left(1-\sum\limits_{j=1}^{n-l}t_j\right)^{\frac{n-1}{2}+l(\beta+1)} \prod\limits_{j=1}^{n-l} t_j ^{\beta} \, dt_{n-l} \ldots dt_1.
\end{align}

Let us prove \eqref{b8} for $l=1.$
\begin{align*}
   & \underbrace{\int\limits_0^{+\infty}\ldots\int\limits_0^{+\infty}}_{n} {\bf 1}\Big{\{} \sum\limits_{j=1}^{n} t_j<1 \Big{\}}   \left(1-\sum\limits_{j=1}^{n}t_j\right)^{\frac{n-1}{2}} \prod\limits_{j=1}^{n} t_j ^{\beta} \, dt_n \ldots dt_1 \\
    &=\underbrace{\int\limits_0^{+\infty}\ldots\int\limits_0^{+\infty}}_{n-1} {\bf 1} \Big{\{} \sum\limits_{j=1}^{n-1} t_j<1 \Big{\}} \prod\limits_{j=1}^{n-1} t_j^{\beta} \int\limits_0^{1-\sum\limits_{j=1}^{n-1} t_j} \left(1-\sum\limits_{j=1}^{n} t_j \right)^{\frac{n-1}{2}}  t_n^{\beta} \, dt_n \ldots dt_1\\
    \end{align*}
    We change variable $t_n$ to the variable $x_n=\frac{t_n}{1-\sum\limits_{j=1}^{n-1} t_j}.$ Then we see that our quantity is equal to
\begin{align}
\label{b9}
  & =\underbrace{\int\limits_0^{+\infty}\ldots\int\limits_0^{+\infty}}_{n-1} {\bf 1}\Big{\{} \sum\limits_{j=1}^{n-1} t_j<1 \Big{\}} \left(1-\sum\limits_{j=1}^{n-1} t_j\right)^{\frac{n-1}{2}+\beta}\prod\limits_{j=1}^{n-1} t_j^{\beta}  \int\limits_0^{1} \left(1- x_n \right)^{\frac{n-1}{2}}  x_n^{\beta} \, dx_n dt_{n-1} \ldots dt_1.
    \end{align}
Note that we have 
\begin{align}
\label{b10}
    \int\limits_0^{1} \left(1- x_n \right)^{\frac{n-1}{2}}  x_n^{\beta} \, dx_n=B\left(\frac{n-1}{2}+1, \beta+1\right)=\frac{\Gamma(\frac{n-1}{2}+1)\Gamma\left(\beta+1\right)}{\Gamma\left(\frac{n-1}{2}+\beta+2\right)}.
\end{align}
We substitute \eqref{b10} to \eqref{b9} and obtain \eqref{b8} with $l=1.$ 

Suppose that \eqref{b8} holds for some $l,$ let us prove that \eqref{b8} holds for $l+1.$ Similarly to the case $l=1$ we change variable $t_{n-l}$ to variable $x_{n-l}=t_{n-l}\left(1-\sum\limits_{j=1}^{n-l-l}t_j\right)^{-1}$ and write
\begin{align}
\label{b11}
     &\underbrace{\int\limits_0^{+\infty}\ldots\int\limits_0^{+\infty}}_{n-l}{\bf 1}\Big{\{} \sum\limits_{j=1}^{n-l} t_j<1 \Big{\}}   \left(1-\sum\limits_{j=1}^{n-l}t_j\right)^{\frac{n-1}{2}+l(\beta+1)} \prod\limits_{j=1}^{n-l} t_j ^{\beta} \, dt_{n-l} \ldots dt_1\\
     \nonumber
      &=\underbrace{\int\limits_0^{+\infty}\ldots\int\limits_0^{+\infty}}_{n-l-1}{\bf 1}\Big{\{} \sum\limits_{j=1}^{n-l-1} t_j<1 \Big{\}}   \left(1-\sum\limits_{j=1}^{n-l-1}t_j\right)^{\frac{n-1}{2}+l(\beta+1)+\beta+1} \prod\limits_{j=1}^{n-l-1} t_j ^{\beta} \\
      \nonumber
      & \int\limits_0^{1} \left(1- x_{n-l} \right)^{\frac{n-1}{2}+l\left(\beta+1\right)}  x_{n-l}^{\beta}  \, dx_{n-l}\, dt_{n-l-1}\ldots dt_1. 
\end{align}
Note that 
\begin{align}
\nonumber
   &\int\limits_0^{1} \left(1- x_{n-l} \right)^{\frac{n-1}{2}+l\left(\beta+1\right)}  x_{n-l}^{\beta}  \, dx_{n-l} =B\left(\frac{n-1}{2}+l\left(\beta+1\right)+1, \beta+1\right)\\
   \label{b12}
   &=\frac{\Gamma\left(\frac{n-1}{2}+l\left(\beta+1\right)+1\right) \, \Gamma\left(\beta+1\right)}{\Gamma\left(\frac{n-1}{2}+\left(l+1\right)\left(\beta+1\right)+1\right)}.
\end{align}
By substituting \eqref{b12} to \eqref{b11} we obtain that \eqref{b8}  holds for $l+1.$ Therefore, the formula \eqref{b8} is proved.
By substituting \eqref{b8} to \eqref{b13} we obtain \eqref{b7}.
\end{proof}
By Lemma \ref{ll3}, \eqref{lemm1} and \eqref{b14} we obtain \eqref{b4} and finish the proof of Lemma \ref{ll2}.
\end{proof}

By \eqref{sum} and Lemma \ref{ll2}  we obtain Proposition \ref{ll1}

\subsection{Proof of Proposition \ref{ll5}}
Let us introduce the following notation:
\begin{align*}
    &\varphi_i=\angle U_1OU_{i} \text{ for } i \in \{2, \ldots, 2n-r-1\},\\
& \gamma_i = \angle U_{n-r+1}OU_{i} \text{ for } i \in \{n-r+1, \ldots, 2n-r\},\\
&\rho_i=1-\|OU_i\|  \text{ for } i \in \{1, \ldots, 2n-r\}.
\end{align*}

Such   notation corresponds to  (\ref{bb1}) and (\ref{bet1}) for  each $i \in \{1, \ldots, n-1\}.$
It is clear that $\gamma_i=(\varphi_i-\varphi_{n-r}) \mod{2\pi}$ for any  $i \ge n.$
We introduce the  events
\begin{align*}
Q_{i,j}= &\{\|V_i^{\varphi}- (\varphi_2, \ldots, \varphi_{n})\| \le \bar{D}\sqrt{\varepsilon},  \| (\rho_1, \ldots, \rho_{n})\| \le \bar{D}\varepsilon, \\ &\|V_j^{\varphi}- (\gamma_{n-r+1}, \ldots, \gamma_{2n-r})\| \le \bar{D}\sqrt{\varepsilon}, \| (\rho_{n-r+1}, \ldots, \rho_{2n-r})\| \le \bar{D}\varepsilon\},
 \end{align*}
 where $V_i^{\varphi}$ is the same as in \eqref{b17}   and constant $\bar{D}$ is introduced in Lemma \ref{ll4}.
 By Lemma \ref{ll4}  for small $ z_N (t) $  the following equality holds:
\begin{align}
&\nonumber\{ h(U_1, \ldots, U_n)\ge z_N(t)\cap h(U_{1+n-r}, \ldots, U_{2n-r})\ge z_N(t) \} \\
&=\cup_{ 1\le i,j \le k} \left(\left[h(U_1, \ldots, U_n)\ge z_N(t)\cap h(U_{1+n-r}, \ldots, U_{2n-r})\ge z_N(t)\right] \cap Q_{i,j}\right).
\label{f20}
\end{align}

Next, we estimate the  probability
 \begin{align}
\label{pp1}
\mathbb{P} \left[\left(f(U_1, \ldots, U_n)\ge z_N(t)\cap f(U_{1+n-r}, \ldots, U_{2n-r})\ge z_N(t)\right) \cap Q_{i,j} \right].
\end{align}
By  definition  for all elements  $V_i$ from $Q_{i,j}$  we have the following bounds for $\varphi_i,\gamma_i$ and $\rho_i:$ $\|\varphi_{l+1}-V_i^l\| \le \bar{D}\sqrt{\varepsilon}$ for each  $i\le n,$   $\|\gamma_{l+1}-V_j^{l-n+r}\|\le \bar{D}\sqrt{\varepsilon},$ and $\rho_i<\bar{D}\varepsilon$ for $i\le 2n-r.$
For $ l \ge n $ we obtain 
\begin{align*}
    &\left\|\varphi_{l+1}-V_i^{n-r}-V_j^{l-n+r}\right\| \le \left\|\varphi_{l+1}-\varphi_{n-r+1}-V_j^{l-n+r}\right\|+\left\|\varphi_{n-r+1}-V_i^{n-r}\right\| \\&\le 2\bar{D}\sqrt{\varepsilon}.
\end{align*}
Using the properties of distribution   of $\varphi_{l+1},$  we can estimate  the upper bound of probability $(\ref{pp1})$ by 
\begin{align*}
    \underbrace{\int\limits_{-2\bar{D}\sqrt{\varepsilon}}^{2\bar{D}\sqrt{\varepsilon}}\ldots  \int\limits_{2\bar{D}\sqrt{\varepsilon}}^{2\bar{D}\sqrt{\varepsilon}}}_{2n-r-1}\underbrace{\int\limits_0^{\bar{D}\varepsilon}\ldots \int\limits_{0}^{\bar{D}\varepsilon}}_{2n-r} \, \prod\limits_{j=1}^{2n-r} \left(\frac{(\beta+1)}{\pi}\,  r_j^{\beta} (1-r_j) (2-r_j)^{\beta} \right)dr_1\ldots dr_{2n-r}\,  d\phi_2\ldots d\phi_{2n-r}\\ \le \left(4\sqrt{\varepsilon}\right)^{2n-r-1} \left(\frac{ 2^{\beta}(\beta+1)}{\pi}\int_0^{\bar{D}\varepsilon} x^{\beta}\, dx\right)^{2n-r} =
    O\left(\varepsilon^{\frac{2n-r-1}{2}+(2n-r)(\beta+1))}\right).
\end{align*}
Using   formula $(\ref{f20})$ and substituting $\varepsilon=tN^{-\frac{n}{(\beta+3/2)n-1/2}}$ in the estimate of the quantity (\ref{pp1}) we obtain the  inequality
\begin{multline*}
N^{2n-r}\mathbb{P}\left[f(U_1, \ldots, U_n)\ge z_N(t), f(U_{1+n-r}, \ldots, U_{2n-r})\ge z_N(t)\right]\\
 \le N^{2n-r}  k^2 O\left(\left(tN^{-\frac{n}{\left(\left(\beta+3/2\right)n-1/2\right)}}\right)^{\frac{2n-r-1}{2}+(2n-r)(\beta+1))}\right) =
 O(N^{\frac{r-n}{ (2\beta+3)n-1}})= O(N^{\frac{-1}{(2\beta+3)n-1}}).\,
\end{multline*}

\section{Acknowledgments}
 The proof of Theorem 1 was supported by  Ministry of Science and Higher Education of the Russian Federation, agreement No. 075–15–2019–1619, the proof of  Theorem 2 was supported by joint grant RFBR-DFG No. 20-51-12004.

The author would like to thank  Dmitry Zaporozhets  for his
invaluable help concerning this paper.

\bibliographystyle{plain}

\begin{thebibliography}{99}

\bibitem{n7}  A. D. Barbour,  L. Holst,  S. Janson, {\it Poisson  Approximation.} Oxford University Press, London, 1992.

\bibitem{BH70} J. Berman, K. Hanes, {\it Volumes of polyhedra inscribed in the unit sphere in $E^3$.} --- Mathematische Annalen. {\bf 188}, No.~1 (1970), 78--84.

\bibitem{n1} P. R. Halmos, {\it The theory of unbiased estimation.} --- Ann. Math. Statist. {\bf 17} (1946), 34--43.

\bibitem{n26} N. Henze, T. Klein, {\it  The limit distribution of the largest interpoint distance from a symmetric Kotz sample.} --- J. Multivariate Anal. {\bf 57} (1996), 228--239.

\bibitem{n2} W. Hoeffding, {\it A class of statistics with asymptotically normal distribution.} --- Ann. Math. Statist. {\bf 19} (1948), 293--325.

\bibitem{HL16} {\'A}. Horv{\'a}th,  Z. L{\'a}ngi, {\it Maximum volume polytopes inscribed in the unit sphere.} --- Monatshefte f{\"u}r Mathematik. {\bf 281}, No.~2 (2016), 341--354.


\bibitem{n28}  S. R. Jammalamadaka, S. Janson, {\it Asymptotic distribution of the maximum interpoint distance in a sample of random vectors with a spherically symmetric distribution.} --- Ann. Appl. Probab. {\bf 25},  No. 6 (2015), 3571 -- 3591.

\bibitem{n29} Z. Kabluchko, {\it Angles of random simplices and face numbers of random polytopes.} --- Adv. Math. {\bf 380} (2021), article 107612. 


\bibitem{n32} Z. Kabluchko, D.  Temesvari, C. Thäle, {\it Expected intrinsic volumes and facet numbers of random beta-polytopes.} --- Math. Nachr. {\bf 292} (2019), 79–-105. 

\bibitem{n30} Z. Kabluchko, C. Thäle, D. Zaporozhets, {\it Beta polytopes and Poisson polyhedra: f-vectors and angles.} --- Adv. Math. {\bf 374} (2020), article 107333. 

\bibitem{n14}  E. V. Koroleva, Ya. Yu. Nikitin, {\it $U$-max-statistics and limit theorems for perimeters and areas
of random polygons.} --- J. Multivariate Anal.   {\bf 127} (2014),  99--111.

\bibitem{n6}   W. Lao, {\it Some weak limit laws for the diameter of random point sets in bounded regions.} Ph.D. Thesis, Karlsruhe, 2010.

\bibitem{n5}   W. Lao, M. Mayer, {\it $U$-max-statistics.} --- J. Multivariate Anal. {\bf 99} (2008),  2039--2052. 




\bibitem{n4}  A. J. Lee, {\it $U$-statistics: Theory and Practice.}  Routledge, 2019.


\bibitem{n25} P. C. Matthews, A. L. Rukhin, {\it Asymptotic distribution of the normal  sample range.} --- Ann. Appl. Probab. {\bf 3} (1993), 454--466.

\bibitem{n9} M. Mayer,  {\it Random Diameters and Other $U$-max-Statistics.}
 Ph.D. Thesis, Bern University, 2008.
 
\bibitem{n27} M. Mayer,  I. Molchanov, {\it Limit theorems for the diameter of a random
sample in the unit ball.} ---  Extremes {\bf 10} (2007),151--174, 2007.
 
 
\bibitem{n20}  
T. A. Polevaya, Ya. Yu. Nikitin, {\it Limit theorems for areas and perimeters of random inscribed and circumscribed polygons.} --- Zap. Nauchn. Sem. POMI (in Russian), {\bf 486} (2019),  200--213. 


\bibitem{n24} Ya. Yu. Nikitin, E. N. Simarova, {\it Generalized Limit Theorems For U-max Statistics.} --- preprint.
 


\bibitem{n8}   F. B. Silverman, T. Brown, {\it Short distances, flat triangles, and Poisson limits.} --- J. Appl. Probab. {\bf 15} (1978),  815--825.


\bibitem{n21}
E. N. Simarova, {\it Limit Theorems for Generalized Perimeters of Random Inscribed Polygons {I}, } --- Vestnik of Saint Petersburg University. Mathematics. Mechanics. Astronomy. (in {R}ussian) {\bf  7 (65)} , No. 4 (2020), 678--687. [Engl. transl.: Vestnik St. Petersb. Univ. Math.{\bf 53}, No. 4 (2020), 434--442.]

\bibitem{n22} E. N. Simarova, {\it Limit Theorems for Generalized Perimeters of Random Inscribed Polygons {II}} (in {R}ussian). Vestnik of Saint Petersburg University. Mathematics. Mechanics. Astronomy. (in {R}ussian)  {\bf 8 (66)}, No. 1 (2021), 101--110.  [Engl. transl.: Vestnik St. Petersb. Univ. Math. {\bf 54}, No. 1 (2021), 78--85. ]














\end{thebibliography}

\end{document}